\documentclass{article}
\usepackage{amsmath,amsfonts,amsthm,amssymb,amscd}

\usepackage[latin1,applemac]{inputenc}

\usepackage[cyr]{aeguill}

\newcommand{\p}{\partial}
\newcommand{\e}{\varepsilon}

\newcommand{\R}{{\mathbb R}}
\newcommand{\IP}{{\mathbb P}}

\newcommand{\E}{{\mathbb E}}

\newcommand{\BBB}{\boldsymbol{\mathit B}}
\newcommand{\PPP}{\boldsymbol{\mathit P}}

\newcommand{\XXX}{\boldsymbol{\mathit X}}
\newcommand{\IIP}{\boldsymbol{\mathbb P}}

\newcommand{\GGamma}{\boldsymbol\Gamma}

\newcommand{\ttau}{{\boldsymbol\tau}}
\newcommand{\oomega}{{\boldsymbol\omega}}
\newcommand{\OOmega}{{\boldsymbol\Omega}}

\newcommand{\BB}{{\cal B}}
\newcommand{\DD}{{\cal D}}
\newcommand{\EE}{{\cal E}}
\newcommand{\FF}{{\cal F}}

\newcommand{\KK}{{\cal K}}
\newcommand{\LL}{{\cal L}}

\newcommand{\PP}{{\cal P}}
\newcommand{\RR}{{\cal R}}

\newcommand{\XX}{{\cal X}}
\newcommand{\YY}{{\cal Y}}

\newcommand{\RRR}{\boldsymbol\RR}
\newcommand{\FFF}{\boldsymbol\FF}

\newcommand{\PPPP}{{\mathfrak P}}

\newcommand{\nnn}{{\boldsymbol{\mathit n}}}
\newcommand{\uuu}{{\boldsymbol{\mathit u}}}

\newcommand{\diver}{\mathop{\rm div}\nolimits}

\newcommand{\esssup}{\mathop{\rm ess\ sup}}

\theoremstyle{plain}
\newtheorem*{mta}{Theorem A}
\newtheorem*{mtb}{Theorem B}
\newtheorem{theorem}{Theorem}[section]

\newtheorem{proposition}[theorem]{Proposition}

\theoremstyle{definition}
\newtheorem{definition}[theorem]{Definition}

\newtheorem{problem}[theorem]{Problem}
\theoremstyle{remark}

\newtheorem{example}[theorem]{Example}
\newtheorem*{example*}{Example}

\numberwithin{equation}{section}

\begin{document}
\author{Armen Shirikyan\footnote{Department of Mathematics, University of Cergy--Pontoise, CNRS UMR 8088, 2 avenue Adolphe Chauvin, 95302 Cergy--Pontoise, France; E-mail: Armen.Shirikyan@u-cergy.fr}}
\title{Control and mixing for 2D Navier--Stokes equations with space-time localised noise}
\date{\today}
\maketitle
\begin{abstract}
We consider randomly forced 2D Navier--Stokes equations in a bounded domain with smooth boundary. It is assumed that the random perturbation is non-degenerate, and its law is periodic in time and has a  support localised with respect to space and time. Concerning the unperturbed problem, we assume that it is approximately controllable in infinite time by an external force whose support is included in that of the random force. Under these hypotheses, we prove that the Markov process generated by the restriction of solutions to the instants of time proportional to the period possesses  a unique stationary distribution, which is exponentially mixing. The proof is based on a coupling argument, a local controllability property of the Navier--Stokes system, an estimate for the total variation distance between a measure and its image under a smooth mapping, and some classical results from the theory of optimal transport.


\end{abstract}

\tableofcontents

\section{Introduction}
\label{s1}
The main results of this paper can be summarised as follows: first, suitable controllability properties of a non-linear PDE imply the uniqueness and exponential mixing for the associated stochastic dynamics and, second, these properties are satisfied for 2D Navier--Stokes equations with space-time localised noise. To be precise, let us consider from the very beginning the 2D Navier--Stokes system in a bounded domain $D\subset\R^2$ with smooth boundary $\p D$:
\begin{align}
\dot u+\langle u, \nabla\rangle u-\nu\Delta u+\nabla p&=f(t,x), \quad \diver u=0, 
\quad x\in D, \label{1.1}\\
u\bigr|_{\p D}&=0. \label{1.2}
\end{align}
Here $u=(u_1,u_2)$ and~$p$ are unknown velocity and pressure of the fluid,
$\nu>0$ is the viscosity, and~$f$ is an external force. Let us assume that~$f$ is represented as the sum of two functions~$h$ and~$\eta$, the first of which is a given function that is~$H^1$ smooth in space and time and has a locally bounded norm, while the second is either a control or a random force:
\begin{equation} \label{1.3}
f(t,x)=h(t,x)+\eta(t,x). 
\end{equation}
In both cases, we assume that~$\eta$ is sufficiently smooth and bounded, and its restriction to any cylinder of the form $ J_k\times D$ with $ J_k=[k-1,k]$ is localised in both space and time (see below for a more precise description of this hypothesis). Let us denote by~$\nnn$ the outward unit normal to the boundary~$\p D$ and introduce the space 
\begin{equation} \label{1.25}
H=\{u\in L^2(D,\R^2): \diver u=0\mbox{ in $D$}, 
\langle u,\nnn\rangle=0\mbox{ on~$\p D$}\},
\end{equation}
which will be endowed with the usual~$L^2$ norm~$\|\cdot\|$. It is well known that for any $u_0\in H$ problem~\eqref{1.1}, \eqref{1.2} supplemented with the initial condition 
\begin{equation} \label{1.05}
u(0,x)=u_0(x)
\end{equation}
has a unique solution $u=u(t;u_0,f)$, which is a continuous function of time valued in~$H$. 

To state our first result, we introduce a localisation hypothesis on the control~$\eta$. Let $Q\subset  J_1\times D$ be an open set and let $Q_k=\{(t,x):(t-k+1,x)\in Q\}$. 

\begin{description}
\item[\sl Localisation.]
For any integer~$k\ge1$, the restriction of~$\eta$ to the cylinder $ J_k\times D$ is supported by~$Q_k$.
\end{description}

The theorem below on the stabilisation to a non-stationary solution for Navier--Stokes equations is an immediate consequence of a contraction property established in Theorem~\ref{t3.1}. 

\begin{mta}
For any $\rho>0$ and~$\alpha>0$ there is a finite-dimensional subspace $E\subset H_0^1(Q,\R^2)$ and positive constants~$C$ and~$d$ such that the following assertions hold for any functions $\hat u_0\in H$ and $h\in H^1_{\rm loc}(\R_+\times D,\R^2)$ satisfying the inequalities
$$
\|\hat u_0\|\le \rho, \qquad \|h\|_{H^1(J_k\times D)}\le\rho\quad\mbox{for all $k\ge1$}.
$$
\begin{itemize}
\item[\bf(i)]
For any $u_0\in H$ satisfying the condition $\|u_0-\hat u_0\|\le d$ there is a control~$\eta$ such that the restriction of $\eta(t+k-1,x)$ to~$J_1\times D$ belongs to~$E$ for any integer~$k\ge1$ and
\begin{equation} \label{1.5}
\|u(t;u_0,h+\eta)-u(t;\hat u_0,h)\| \le Ce^{-\alpha t}\|u_0-\hat u_0\|, \quad t\ge0. 
\end{equation}
\item[\bf(ii)]
The mapping $(\hat u_0,u_0,h)\mapsto \eta$ is Lipschitz continuous for appropriate norms. Moreover, 
\begin{equation} \label{1.6}
\|\eta\|_{L^2(Q_k,\R^2)}\le Ce^{-\alpha k}\|u_0-\hat u_0\|, \quad k\ge1. 
\end{equation}
\end{itemize}
\end{mta}

Let us note that a similar result was proved earlier in~\cite{BRS-2010} for the 3D Navier--Stokes system (see also~\cite{fursikov-2004,barbu-2003,BT-2004} for some results on stabilisation to a stationary solution). Here we establish some additional properties for the control. Namely, we show that it is finite-dimensional in both space and time and  can be chosen to be a  smooth function with respect to the force and the initial state corresponding to the reference solution. 

\smallskip
We now turn to the stochastic problem. Informally speaking, we assume that the deterministic force~$h$ is $1$-periodic in time, and the stochastic force~$\eta$ satisfies the following three conditions. Let us denote by~$\eta_k$ the restriction of $\eta(t+k-1,x)$ to the domain $J_1\times D$.

\begin{description}
\item[\sl Independence.]
The functions~$\eta_k$ form a sequence of i.i.d.~random variables in~$H^1( J_1\times D,\R^2)$. 
\item[\sl Non-degeneracy.]
There is a compact subset $\KK\subset H_0^1(Q,\R^2)$ such that the support of the law for~$\eta_k$ coincides with~$\KK$, and the law of the projection of~$\eta_k$ to any finite-dimensional subspace of~$L^2(Q,\R^2)$ has a regular density with respect to the Lebesgue measure.
\item[\sl Approximate controllability.]
There is $\hat u\in H$ such that problem \eqref{1.1}--\eqref{1.3} is approximately controllable to~$\hat u$ with a control function~$\tilde\eta$ such that for any $k\ge1$ the restriction of $\tilde\eta(t+k-1,x)$ to~$J_1\times D$ belongs to~$\KK$, and the time of control can be chosen the same for  the initial functions~$u_0$  from a given bounded subset. 
\end{description}
We refer the reader to Section~\ref{s2.1} for the exact description of the hypotheses imposed on the stochastic system~\eqref{1.1}--\eqref{1.3}. Note that the hypothesis on independence implies that the restrictions of solutions to the integer times form a Markov chain in~$H$. The following theorem describes the long-time asymptotics of this chain. 

\begin{mtb}
Under the above hypotheses, the Markov chain associated with problem~\eqref{1.1}--\eqref{1.3} has a unique stationary measure~$\mu$. Moreover, there are positive constants~$C$ and~$\gamma$ such that, for any $1$-Lipschitz function $F:H\to\R$ and any $u_0\in H$, we have 
\begin{equation} \label{1.7}
\biggl|\E \,F\bigl(u(k;u_0,h+\eta)\bigr)-\int_HF(v)\mu(dv)\biggr|
\le C\bigl(1+\|u_0\|\bigr)e^{-\gamma k}, \quad k\ge0.
\end{equation}
\end{mtb}

Let us note that the problem of ergodicity was studied intensively in the last  fifteen years. First results in this direction were established in~\cite{FM-1995,KS-CMP2000,EMS-2001,BKL-2002}, and we refer the reader to the book~\cite{KS2011} for further references. Most of the results established so far concern the situation in which the random force is non-degenerate in all determining modes of the problem. In the case when the equation is studied on the torus and the deterministic force is zero, it was proved in~\cite{HM-2006,HM-2008} that the Navier--Stokes dynamics is exponentially mixing for any $\nu>0$, provided that the noise is white in time and has a few non-zero Fourier modes as a function of~$x$ (thus it is finite-dimensional in space and infinite-dimensional in time). Theorem~B applies to all $\nu>0$, and the random noise does not act directly on the deterministic modes. However, in contrast to~\cite{HM-2006,HM-2008}, our noise is localised not in the Fourier modes with respect to~$x$, but in the physical space and in time. Furthermore, our approach does not use the Malliavin calculus and is based on a detailed study of controllability properties for the Navier--Stokes system and a result on the image of measures on a Hilbert space under finite-dimensional transformations. To the best of my knowledge, Theorem~B provides a first result on mixing properties for Navier--Stokes equations with a space-time localised noise. 

In conclusion, let us mention that Theorems~A and~B remain valid in the case when the control and the noise act through the boundary of the domain. This situation will be addressed in a subsequent publication. 

\smallskip
The paper is organised as follows. In Section~\ref{s2}, we formulate the main result of this paper on exponential mixing for the Navier--Stokes system with space-time localised noise and outline its proof. Section~\ref{s3} is devoted to studying a control problem associated with the stochastic system in question. The details of proof of the main result are given in Section~\ref{s4}. The appendix gathers some auxiliary results used in the main text. 

\subsection*{Notation}
For an open set~$Q$ of a Euclidean space, a closed interval~$J\subset\R$, and Banach spaces~$X\subset Y$, we introduce the following function spaces. 

\smallskip
\noindent
$L^p=L^p(Q)$ is the Lebesgue space of measurable scalar or vector functions on~$Q$ whose $p^{\text{th}}$ power is integrable. We shall sometimes write~$L^p(Q,\R^d)$ to emphasise the range of functions. In the case $p=2$, the corresponding norm will be denoted by~$\|\cdot\|$. 

\smallskip
\noindent
$H^s=H^s(Q)$ is the Sobolev space of order~$s$ with the usual norm~$\|\cdot\|_s$. As in the previous case, we use the same notation for spaces of scalar and vector functions. 

\smallskip
\noindent
$H_0^s=H_0^s(Q)$ is the closure in~$H^s$ of the space of infinitely smooth functions with compact support.

\smallskip
\noindent
$B_X(R)$ stands for the ball in~$X$ of radius~$R$ centred at zero. 

\smallskip
\noindent
$L^p(J,X)$ is the space of Borel-measurable functions $u:J\to X$ such that 
$$
\|u\|_{L^p(J,X)}=\biggl(\int_J\|u(t)\|_X^pdt\biggr)^{1/p}<\infty\,;
$$
in the case $p=\infty$, this norm should be replaced by $\|u\|_\infty=\esssup_J\|u(t)\|_X$.

\smallskip
\noindent
$W^{1,p}(J,X)$ is the space of functions $u\in L^p(J,X)$ whose derivative belongs to $L^p(J,X)$. It is endowed with a natural norm. 

\smallskip
\noindent
$W(J,X,Y)$ is the space of functions $u\in L^2(J,X)$ such that $\p_tu\in L^2(J,Y)$. 

\smallskip
\noindent
$\LL(X,Y)$ is the space of continuous linear operators from~$X$ to~$Y$ with the natural norm. In the case $X=Y$, we write $\LL(X)$. 

\smallskip
\noindent
$C_b(X)$ stands  for the space of bounded continuous functions $F:X\to\R$; it is endowed with the norm
$$
\|F\|_\infty=\sup_{u\in X}|F(u)|.
$$

\smallskip
\noindent
$L_b(X)$ denotes the space of functions $F\in C_b(X)$ such that 
$$
\|F\|_L:=\|F\|_\infty+\sup_{u\ne v}\frac{|F(u)-F(v)|}{\|u-v\|_X}<\infty.
$$

\smallskip
\noindent
$\PP(X)$ is the set of probability Borel measures on~$X$. The space~$\PP(X)$ is endowed with the topology of weak convergence, which is generated by the dual-Lipschitz metric
$$
\|\mu_1-\mu_2\|_L^*:=\sup_{\|F\|_L\le 1}|(F,\mu_1)-(F,\mu_2)|, 
$$
where $(F,\mu)$ stands for the integral of~$F$ over~$X$ with respect to~$\mu$.

\medskip
\noindent
We denote by~$D$ a bounded domain with $C^2$ boundary~$\p D$.  For $T>0$, we set $J_T=[0,T]$ and $D_T=J_T\times D$. 
The following functional spaces arise in the theory of Navier--Stokes equations:
\begin{gather*}
V=H\cap H_0^1(D),\quad
\XX_T=W(J_T,V,V^*), \\
\YY_{\delta,T}=\bigl\{u\in\XX_T: u|_{(\delta,T)}\in W((\delta,T),V\cap H^3,V)\bigr\},
\end{gather*}
where $\delta\in (0,T)$, $H$ is defined by~\eqref{1.25}, and~$V^*$ denotes the dual space of~$V$ (identified with a quotient space in~$H^{-1}(D,\R^2)$ with the help of the scalar product in~$L^2$). These spaces are endowed with natural norms. 

\section{Main result and scheme of its proof}
\label{s2}
\subsection{Exponential mixing}
\label{s2.1}
Let $D\subset\R^2$ be a bounded domain with a $C^2$-smooth boundary~$\p D$ and let $D_1=J_1\times D$. Consider the Navier--Stotes system~\eqref{1.1} with the Dirichlet boundary condition~\eqref{1.2} and an external force of the form~\eqref{1.3}. We assume that $h\in H_{\rm loc}^1(\R_+\times D,\R^2)$ is a given function which is $1$-periodic in time and~$\eta$ is a stochastic process of the form
\begin{equation} \label{2.3}
\eta(t,x)=\sum_{k=1}^\infty I_k(t)\eta_k(t-k+1,x), \quad t\ge0,
\end{equation}
where $I_k$ is the indicator function of the interval $(k-1,k)$ and~$\{\eta_k\}$ is a sequence of i.i.d.~random variables in $L^2(D_1,\R^2)$ that are continued by zero for $t\notin J_1$. It is well known that the Cauchy problem for~\eqref{1.1}--\eqref{1.3} is globally well posed. Namely, for any initial function~$u_0\in H$ there is a unique random process $(u,p)$ whose almost every trajectory satisfies the inclusions
$$
u\in C(\R_+,H)\cap L_{\rm loc}^2(\R_+,V), \quad p\in L_{\rm loc}^1(\R_+,L^2),
$$
Equations~\eqref{1.1}, and the initial condition
\begin{equation} \label{2.4}
u(0,x)=u_0(x), \quad x\in D.
\end{equation}
In what follows, we shall drop the~$p$ component of solutions and write simply~$u(t)$. Let us denote by $S:H\times L^2(D_1,\R^2)\to H$ the operator that takes a pair of functions $(u_0,f)$ to~$u(1)$, where $u(t)$ is the solution of~\eqref{1.1}, \eqref{1.2}, \eqref{2.4}. Well-known properties of 2D Navier--Stokes equations imply that~$S$ is a continuous mapping. Moreover, the range of~$S$ is contained in~$V$, and the mapping $S:H\times L^2(D_1,\R^2)\to V$ is uniformly Lipschitz continuous on bounded subsets; e.g., see Chapter~III in~\cite{Te}. 

Let us consider a solution $u(t)$ of~\eqref{1.1}--\eqref{1.3} and denote $u_k=u(k)$. What has been said implies that 
\begin{equation} \label{2.5}
u_k=S(u_{k-1},h+\eta_k), \quad k\ge1.
\end{equation}
Since~$\eta_k$ are i.i.d.~random variables in~$L^2(D_1,\R^2)$,  Eq.~\eqref{2.5} defines a homogeneous family of Markov chains in~$H$, which is denoted by~$(u_k,\IP_u)$, $u\in H$. Let $P_k(u,\Gamma)$ be the transition function for the family~$(u_k,\IP_u)$ and let $\PPPP_k:C_b(H)\to C_b(H)$ and $\PPPP_k^*:\PP(H)\to\PP(H)$ be the corresponding Markov semigroups. 

Let us fix an open set $Q\subset D_1$ and denote by~$\{\varphi_j\}\subset H^1(Q,\R^2)$ an  orthonormal basis in~$L^2(Q,\R^2)$.  Let $\chi\in C_0^\infty(Q)$ be a non-zero function and let $\psi_j=\chi\varphi_j$. In what follows, we shall assume that~$\{\psi_j\}$ are linearly independent\footnote{The assumption on the linear independence of~$\{\psi_j\}$ is not really needed, and Theorem~\ref{t2.1} below remains true without it. We make, however, this assumption to simplify the proof of Proposition~\ref{p2.3}.},  and the function~$h$ and random process~$\eta$ satisfy hypotheses~(H1) and~(H2) formulated below. Note that if $\{\varphi_j\}$ is a complete set of eigenfunctions of the Dirichlet Laplacian in~$Q$, then the functions~$\psi_j$ are linearly independent for any choice of~$\chi$. This is an immediate consequence of the unique continuation property of solutions for elliptic equations; see Theorem~8.9.1 in~\cite{hormander1963}.

\begin{description}
\item[(H1)\ Structure of the noise.]
The random variables~$\eta_k$ can be represented in the form
\begin{equation} \label{2.6}
\eta_k(t,x)=\sum_{j=1}^\infty b_j\xi_{jk}\psi_j(t,x),
\end{equation}
where $\xi_{jk}$ are independent scalar random variables such that $|\xi_{jk}|\le1$ with probability~$1$, and $\{b_j\}\subset\R$ is a non-negative sequence such that 
\begin{equation} \label{2.7}
B:=\sum_{j=1}^\infty b_j\|\psi_j\|_1<\infty.
\end{equation}
Moreover, the law of~$\xi_{jk}$ possesses a $C^1$-smooth density~$\rho_j$ with respect to the Lebesgue measure on the real line.
\end{description}

Let us denote by~$\KK\subset L^2(Q,\R^2)$ the support of the law of~$\eta_k$. The hypotheses imposed on~$\eta_k$ imply that~$\KK$ is a compact subset in $H_0^1(Q,\R^2)$. Continuing the elements of~$\KK$ by zero outside~$Q$, we may regard~$\KK$ as a compact subset of~$H_0^1(D_1,\R^2)$.

\begin{description}
\item[(H2)\ Approximate controllability.]
There is $\hat u\in H$ such that for any positive constants~$R$ and~$\e$ one can find an integer $l\ge1$ with the following property: given $v\in B_H(R)$, one can find $\zeta_1, \dots,\zeta_l\in\KK$ such that 
\begin{equation} \label{2.8}
\|S_l(v,\zeta_1,\dots,\zeta_l)-\hat u\|\le\e,
\end{equation}
where $S_l(v,\zeta_1,\dots,\zeta_l)$ stands for the vector~$u_l$ defined  by~\eqref{2.5} with $\eta_k=\zeta_k$ and~$u_0=v$. 
\end{description}

The following theorem, which is the main result of this paper, establishes the uniqueness and exponential mixing of a stationary distribution for the Markov family generated by~\eqref{2.5}.

\begin{theorem} \label{t2.1}
Assume that $h\in H_{\rm loc}^1(\R_+\times D,\R^2)$ is $1$-periodic in time, and Conditions~{\rm(H1)} and~{\rm(H2)} are satisfied. In this case, there is an integer~$N\ge1$, depending on~$\|h\|_{H^1(D_1)}$, $B$, and~$\nu$, such that if 
\begin{equation} \label{2.9}
b_j\ne0\quad\mbox{for $j=1,\dots,N$},
\end{equation}
then the following assertions hold.
\begin{description}
\item[Existence and uniqueness.]
The Markov family $(u_k,\IP_u)$ has a unique stationary distribution $\mu\in\PP(H)$. 
\item[Exponential mixing.]
There are positive constants~$C$ and~$\gamma$ such that 
\begin{equation} \label{2.10}
\|P_k(u,\cdot)-\mu\|_L^*\le C(1+\|u\|)e^{-\gamma k}
\quad\mbox{for all $u\in H$, $k\ge0$}. 
\end{equation}
\end{description}
\end{theorem}
Note that condition~\eqref{2.9} expresses the space-time non-degeneracy of the noise. Thus, the property of exponential mixing holds true even for noises whose space-time dimension is finite. 

The general scheme of the proof of Theorem~\ref{t2.1} is outlined in Section~\ref{s2.2}, and the details are given in Sections~\ref{s3} and~\ref{s4}. Here we consider two  examples for which Condition~(H2) is fulfilled. 

\begin{example} \label{e2.2}
We claim that there is $\delta>0$ such that if the $1$-periodic function~$h$ satisfies the inequality  $\|h\|_{L^2(D_1)}\le\delta$, then Condition~(H2) is fulfilled, provided that~$\KK$ contains the zero element. 

Indeed, if $\|h\|_{L^2(D_1)}$ is sufficiently small, then problem~\eqref{1.1}, \eqref{1.2} with $f=h$ has a unique solution~$\tilde u(t,x)$ defined throughout the real line and $1$-periodic in time. To see this, it suffices to take a sequence~$\{u_n\}$ of solutions for the problem in question such that $u_n(-n)=0$ and to prove that it converges as $n\to\infty$ in the space $W([-N,N],V,V^*)$ for any $N>0$. The limiting function~$\tilde u$ is the required $1$-periodic solution. Using standard estimates for Navier--Stokes equations, it is easy to prove that 
$$
\|\tilde u\|_{W(J_1,V,V^*)}\le c(\delta)\to0\quad\mbox{as $\delta\to0$}.
$$
This implies that~$\tilde u$ is globally exponentially  stable as $t\to+\infty$. Therefore, for any positive constants~$R$ and~$\e$ one can find an integer~$l\ge1$ such that~\eqref{2.8} holds with $\zeta_1=\cdots=\zeta_l=0$ and $\hat u=\tilde u(0)$.  Since~$\KK$ contains the zero element, we see that Condition~(H2) is satisfied. 
\end{example}

\begin{example}
Suppose that $h$ is represented in the form $h(t,x)=h_0(\lambda t,x)$, where $h_0$ is a continuous $1$-periodic function of time  with range in~$V$ such that $\int_0^1h_0(s)\,ds=0$. We claim that, for sufficiently large integers $\lambda>0$, Condition~(H2) is satisfied, provided that~$\KK$ contains the zero element. 

Indeed, let us represent a solution of Eq.~\eqref{1.1} with $f=h$ in the form $u=w+g$, where $g(t)=\int_0^th(s)\,ds$. Then the function~$w$ must satisfy the equations
\begin{equation}\label{2.1}
\dot w+\langle w+g, \nabla\rangle (w+g)-\nu\Delta w+\nabla p=-\Delta g, \quad \diver w=0.
\end{equation}
The condition imposed on~$h$ imply that~$g$ is a $V$-valued $1$-periodic function such that
$$
\sup_{t\in\R}\|g(t)\|_V\to0\quad\mbox{as $\lambda\to\infty$}. 
$$
Combining this with an argument similar to that used in Example~\ref{e2.2}, one can  prove that Eq.~\eqref{2.1} has a unique $1$-periodic solution~$\widetilde w$, which is globally exponentially stable as $t\to\infty$. It follows that $\tilde u=\widetilde w+g$ is a globally exponentially stable solution for problem~\eqref{1.1}, \eqref{1.2} with $f=h$. As in Example~\ref{e2.2}, we conclude that the function~$\hat u=\tilde u(0)$ satisfies the required property. 
\end{example}

\subsection{General criterion for mixing and application}
\label{s2.2}
In this section, we outline the proof of Theorem~\ref{t2.1}, which is based on two key ingredients: a coupling approach developed in~\cite{KS-CMP2001,mattingly-2002,KPS-CMP2002,MY-2002,hairer-2002,shirikyan-JMFM2004} in the context of stochastic PDE's and a property of stabilisation to a non-stationary solution of Navier--Stokes equations~\cite{BRS-2010}. We first recall an abstract result established in~\cite{shirikyan-2008}.

Let $X$ be a compact metric space with a metric~$d_X$ and let~$(u_k,\IP_u)$, $u\in X$, be a family of Markov chains in~$X$. We denote by $P_k(u,\Gamma)$ its transition function and by~$\PPPP_k$ and~$\PPPP_k^*$ the corresponding Markov semigroups.  Let~$(\uuu_k,\IP_\uuu)$ be another family of Markov chains in the extended phase space $\XXX=X\times X$ such that 
\begin{equation} \label{2.21}
\Pi_*\PPP_k(\uuu,\cdot)=P_k(u,\cdot), \quad
\Pi_*'\PPP_k(\uuu,\cdot)=P_k(u',\cdot)\quad
\mbox{for $\uuu=(u,u')\in\XXX$, $k\ge0$},
\end{equation}
where $\PPP_k(\uuu,\GGamma)$ denotes the transition function for~$(\uuu_k,\IP_\uuu)$ and $\Pi,\Pi':\XXX\to X$ stand for the natural projections to the components of a vector $\uuu=(u,u')$. In other words, relations~\eqref{2.21} mean that, for any integer~$k\ge1$, the random variable $\uuu_k$ considered under the law~$\IP_\uuu$ with $\uuu=(u,u')$ is a coupling for the pair of measures $(P_k(u,\cdot),P_k(u',\cdot))$. 
We shall say that~$(\uuu_k,\IP_\uuu)$ satisfies the {\it mixing hypothesis\/} if there is a closed subset $\BBB\subset \XXX$ and positive constants~$C$ and~$\beta$ such that the following properties hold. 

\begin{description}
\item[Recurrence.]
Let $\ttau(\BBB)$ be the first hitting time of the set~$\BBB$:
$$
\ttau(\BBB)=\min\{k\ge0:\uuu_k\in\BBB\}.
$$
Then $\ttau(\BBB)$ is $\IP_\uuu$-almost surely finite for any $\uuu\in\XXX$, and there are positive constants~$C_1$ and~$\delta_1$ such that
\begin{equation} \label{2.22}
\E_\uuu \exp\bigl(\delta_1\ttau(\BBB)\bigr)\le C_1 
\quad \mbox{for $\uuu\in\XXX$}.
\end{equation}

\item[Exponential squeezing.]
Let us set
\begin{equation} \label{2.23}
\sigma=\min\{k\ge0:d_X(u_k,u_k')>C\,e^{-\beta k}\}. 
\end{equation}
Then there are positive constants $C_2$, $\delta_2$, and~$\delta_3$ such that, for any $\uuu\in\BBB$, we have
\begin{align}
\IP_\uuu\{\sigma=\infty\}&\ge\delta_3,\label{2.24}\\
\E_\uuu\bigl(I_{\{\sigma<\infty\}}\exp(\delta_2\sigma)\bigr)&\le C_2.\label{2.25}
\end{align}
\end{description}
The following proposition is a particular case of a more general result established\footnote{In~\cite{shirikyan-2008}, the proof is carried out in the particular case when $\BBB=B\times B$; however, the same argument applies in the general situation. } in~\cite{shirikyan-2008} (see Theorem~2.3). 

\begin{proposition} \label{p2.2}
Let $(u_k,\IP_u)$ be a family of Markov chains for which there exists another Markov family $(\uuu_k,\IP_\uuu)$ in the extended space~$\XXX$ that satisfies relation~\eqref{2.21} and the mixing hypothesis. Then~$(u_k,\IP_u)$ has a unique stationary distribution $\mu\in\PP(X)$, and there are positive constants~$C$ and~$\gamma$ such that 
\begin{equation} \label{2.26}
\|P_k(u,\cdot)-\mu\|_L^*\le C\,e^{-\gamma k}\quad\mbox{for $u\in X$, $k\ge0$}.
\end{equation}
\end{proposition}

To prove Theorem~\ref{t2.1}, we first observe that the Markov family~$(u_k,\IP_u)$ possesses a compact absorbing invariant set~$X\subset H$, and it suffices to study its restriction to~$X$, for which we retain the same notation. 
We shall prove that~$(u_k,\IP_u)$  satisfies the hypotheses of Proposition~\ref{p2.2}. A crucial point of our construction is the following result, which says, roughly speaking, that if two points $u,u'\in X$ are sufficiently close, then the pair $(P_1(u,\cdot),P_1(u',\cdot))$ admits a coupling whose components are close with high probability; cf.\ Lemma~3.3 in~\cite{KS-CMP2001}.  

\begin{proposition} \label{p2.3}
Under the hypotheses of Theorem~\ref{t2.1}, there exists a constant $d>0$ such that for any points $u,u'\in X$ satisfying the inequality $\|u-u'\|\le d$ the pair $(P_1(u,\cdot),P_1(u',\cdot))$ admits a coupling $(V(u,u'),V'(u,u'))$ such that 
\begin{equation} \label{2.27}
\IP\bigl\{\|V(u,u')-V'(u,u')\|>\tfrac12\|u-u'\|\bigr\}\le C\,\|u-u'\|,
\end{equation}
where $C>0$ is a constant not depending on $u,u'\in X$. 
\end{proposition}

The proof of this proposition is based on a controllability property for the Navier--Stokes system and application of a concept of optimal coupling; see Sections~\ref{s3} and~\ref{s5.1}. We now define a {\it coupling operator\/}~$\RRR=(\RR,\RR')$ by the relation
\begin{equation} \label{2.28}
\RRR(u,u';\omega)=\left\{
\begin{array}{cl}
\bigl(V(u,u'),V'(u,u')\bigr)&\mbox{for $\|u-u'\|\le d$},\\[3pt]
\bigl(S(u,\zeta),S(u',\zeta')\bigr)&\mbox{for $\|u-u'\|> d$},
\end{array}
\right.
\end{equation}
where $\zeta$ and~$\zeta'$ are independent random variables whose law coincides with that of~$\eta_1$. Without loss of generality, we can assume that~$\zeta$ and~$\zeta'$ are defined on the same probability space as~$V$ and~$V'$. The required Markov family~$(\uuu_k,\IP_\uuu)$ is constructed by iterations of~$\RRR$. Namely, let~$(\Omega_k,\FF_k,\IP_k)$, $k\ge1$, be countably many copies of the probability space on which~$\RRR$ is defined and let $(\OOmega,\FFF,\IIP)$ be the direct product of these spaces. We set 
\begin{equation} \label{2.29}
\uuu_0=(u,u'), \quad \uuu_k=\RRR(\uuu_{k-1},\omega_k), \quad k\ge1. 
\end{equation}
The recurrence property will follows from approximate controllability (see Hypothesis~(H2) in Section~\ref{s2.1}), while the exponential squeezing will be implied by Proposition~\ref{p2.3}. 

\section{Control problem}
\label{s3}
\subsection{Squeezing}
\label{s3.1}
In this section, we consider the controlled Navier--Stokes system~\eqref{1.1}--\eqref{1.3} on the time interval $J_1=[0,1]$. We assume that the function~$h$ belongs to the space $H^1(D_1)$ (where $D_1=J_1\times D$) and denote by~$\{\psi_j\}$ the sequence of functions entering Hypothesis~(H1) of Section~\ref{s2.1}. Extending the functions~$\psi_j$ by zero outside~$Q$, we may regard them as elements of $H_0^1(D_1)$. We denote by~$\EE_m$ the vector span of $\psi_1,\dots,\psi_m$ endowed with the~$L^2$ norm and by~$B_R$ the ball in~$H^1(D_1)$ of radius~$R$ centred at origin.

\begin{theorem} \label{t3.1}
Under the above hypotheses, for any $R>0$ and $q\in(0,1)$ there is an integer $m\ge1$, positive constants~$d$ and~$C$, and a continuous mapping 
$$
\varPhi:B_R\times B_H(R)\to \LL(H,\EE_m),\quad
(h,\hat u_0)\mapsto \eta, 
$$ 
such that the following properties hold.

\begin{description}
\item[Contraction.]
For any functions $h\in B_R$ and $\hat u_0, u_0\in B_H(R)$ satisfying the inequality
\begin{equation} \label{3.1}
\|u_0-\hat u_0\|\le d, 
\end{equation}
we have
\begin{equation} \label{3.2}
\bigl\|S\bigl(\hat u_0,h\bigr)
-S\bigl(u_0,h+\varPhi(h,\hat u_0)(u_0-\hat u_0)\bigr)\bigr\|
 \le q\,\|u_0-\hat u_0\|. 
\end{equation}
\item[Regularity.]
The mapping~$\varPhi$ is infinitely smooth in the Fr\'echet sense. 
\item[Lipschitz continuity.]
The mapping $\varPhi$ is Lipschitz continuous with the constant~$C$. That is, 
\begin{equation} \label{3.3}
\bigl\|\varPhi(h_1,\hat u_1)-\varPhi(h_2,\hat u_2)\bigr\|_{\LL}\le C\,\bigl(\|h_1-h_2\|_{H^1}+\|\hat u_1-\hat u_2\|\bigr),
\end{equation}
where $\|\cdot\|_\LL$ stands for the norm in the space $\LL(H,\EE_m)$. 
\end{description}
\end{theorem}

Taking this result for granted, let us prove Theorem~A stated in the Introduction. 
Let us fix positive constants~$\rho$ and~$\alpha$ and take an initial function $\hat u_0\in B_H(\rho)$. The boundedness of the resolving operator for the Navier--Stokes system implies that the corresponding solution~$u$ satisfies the inequality
\begin{equation} \label{3.51}
\|u(t)\|\le R\quad\mbox{for all $t\ge0$},
\end{equation}
where~$R$ is a constant depending only on~$\rho$. Let $q\in(0,1)$ be such that $e^{-\alpha}=q$. Denote by~$d$, $C$, and~$m$ the parameters constructed in Theorem~\ref{t3.1}. Let us take any $u_0\in H$ satisfying the inequality~$\|u_0-\hat u_0\|\le d$ and construct a control function~$\eta$ consecutively on the intervals~$J_k$, $k\ge1$. To this end, we denote by~$h_k$ the restriction of the function $h(t-k+1)$ to~$D_1$ and set
$$
\eta(t)=\varPhi(\hat u_0,h_1)(u_0-\hat u_0)\quad\mbox{for $0\le t\le 1$}. 
$$
Then, by Theorem~\ref{t3.1}, we have 
$$
\|u(1)-\hat u(1)\|\le q\|u_0-\hat u_0\|=e^{-\alpha}\|u_0-\hat u_0\|,
$$
where $u(t)$ denotes the solution of~\eqref{1.1}--\eqref{1.3}, \eqref{1.05} on the interval~$J_1$. Assume we have constructed a control~$\eta$ on the interval~$(0,k)$ with $k\ge 1$, and the corresponding solution~$u$ satisfies the inequality
\begin{equation} \label{decay}
\|u(l)-\hat u(l)\|\le e^{-\alpha l}\|u_0-\hat u_0\|
\end{equation}
for $l=1,\dots, k$. In this case, setting
\begin{equation} \label{control}
\eta(k+t)=\varPhi(\hat u(k),h_{k+1})\bigl(u(k)-\hat u(k)\bigr)
\quad\mbox{for $0\le t\le 1$},
\end{equation}
we see that~\eqref{decay} remains valid for  $l=k+1$. Thus, we can construct~$\eta$ on the half-line~$\R_+$, and the corresponding solution will satisfy~\eqref{decay} for all integers $l\ge1$. Combining this with the Lipschitz continuity of the resolving operator of the Navier--Stokes system, we see that~\eqref{1.5} holds. Inequality~\eqref{1.6} is a straightforward consequence of~\eqref{control} and~\eqref{decay}. 
Finally, it is not difficult to check that if~$\eta^1$ and~$\eta^2$ are two controls corresponding to $(u_{0}^i,\hat u_{0}^i,h^i)$, $i=1,2$, then
$$
\bigl\|\eta_k^1-\eta_k^2\bigr\|_{L^2(D_k)}
\le C_1^k\Bigl(\|u_{0}^1-u_{0}^2\|+\|\hat u_{0}^1-\hat u_{0}^2\|
+\max_{1\le l\le k}\|h_{l}^1-h_{l}^2\|_{H^1(D_l)}\biggr),
$$
where $C_1>0$ depends only on~$\rho$ and~$\alpha$, $\eta_{l}^i$ stands for the restriction of~$\eta^i$ to~$D_l$, and~$h_l^i$ are defined in a similar way. This completes the proof of Theorem~A. 

\medskip
We now turn to the proof of Theorem~\ref{t3.1}. 
As was mentioned in the Introduction, a weaker version of this result was established in the paper~\cite{BRS-2010}, and its proof repeats essentially the argument used there. However, since the finite-dimensionality in time for the control and the  regularity and Lipschitz properties of~$\varPhi$ are important for the stochastic part of this work, we present a rather complete proof of Theorem~\ref{t3.1}. We begin with a description of the main steps and give the details in the next two subsections. 

\medskip
{\it Step~1: Reduction to the linearised problem}. 
Denote by~$\hat u(t,x)$ the solution of~\eqref{1.1}, \eqref{1.3} issued from~$\hat u_0$ and corresponding to~$\eta\equiv0$. In view of the regularising property of the Navier--Stokes system, for any interval $J=(\delta,1)$ with $\delta>0$ we have
\begin{equation} \label{reg}
\hat u\in L^2(J,H^3\cap V), \quad \p_t\hat u\in L^2(J,V),
\end{equation}
and the corresponding norms are bounded by a constant depending only on~$\delta$ and~$R$. In particular, the truncated observability inequality holds for the adjoint of the Navier--Stokes system linearised around~$\hat u$; see Section~\ref{s5.2}. 

A solution with a non-zero control~$\eta$ is sought in the form $u=\hat u+v$. Then~$v$ must be a solution of the problem
\begin{gather}
\dot v+\langle v,\nabla\rangle v+\langle \hat u,\nabla\rangle v
+\langle v,\nabla\rangle \hat u-\nu \Delta v+\nabla p=\eta(t,x), \quad \diver v=0,
\label{3.4}\\
v\bigr|_{\p D}=0, \quad v(0)=v_0, \label{3.5}
\end{gather}
where $v_0=u_0-\hat u_0$. Together with Eq.~\eqref{3.4}, consider its linearisation around zero:
\begin{equation}
\dot v+\langle \hat u,\nabla\rangle v
+\langle v,\nabla\rangle \hat u-\nu \Delta v+\nabla p=\eta(t,x), \quad \diver v=0.
\label{3.6}
\end{equation}
Suppose that we have constructed $\eta\in\EE_m$ such that the solution~$w(t,x)$ of~\eqref{3.6}, \eqref{3.5} satisfies the inequalities
\begin{equation} \label{3.7}
\|w(1)\|\le\frac{q}{2}\|v_0\|, \quad \|w\|_{\XX_1}\le C_1\|v_0\|. 
\end{equation}
A standard perturbative argument shows that if~$\|v_0\|$ is sufficiently small, then the solution of~\eqref{3.4}, \eqref{3.5} satisfies the inequality $\|v(1)\|\le q\|v_0\|$, whence it follows that~\eqref{3.2}  holds. Thus, it suffices to construct a continuous linear operator $\varPhi(h,\hat u_0):H\to\EE_m$ such that the solution $w\in\XX_1$ of problem~\eqref{3.6}, \eqref{3.5} with $\eta=\varPhi(h,\hat u_0)v_0$ satisfies inequalities~\eqref{3.7}. 

\smallskip
{\it Step~2: Application of the Foia\c s--Prodi property}. 
Let $\{e_j\}$ be an orthonormal basis in~$H$ formed of the eigenfunctions of the Stokes operator $L=-\Pi\Delta$, where~$\Pi$ stands for the Leray projection in~$L^2(D,\R^2)$ (onto the closed subspace~$H$), let~$\{\alpha_j\}$ be the corresponding (non-decreasing) sequence of eigenvalues for~$L$, and let $\Pi_N$ be the orthogonal projection in~$H$ on the vector space~$H_N$ spanned by  $e_1,\dots,e_N$. 
Denote by~$\RR^{\hat u}:H\times L^2(D_1)\to \XX_1$ a linear operator that takes~$(v_0,\eta)$ to the solution~$w$ of~\eqref{3.6}, \eqref{3.5} and by~$\RR_t^{\hat u}$ its restriction to the time~$t$. 

Suppose that for any integer $N\ge1$ and any~$\delta>0$ we have constructed an integer $m\ge1$ and a family of linear operators $\varPhi=\varPhi(h,\hat u_0):H\to\EE_m$ which is a Lipschitz function of its arguments and is such that
\begin{equation} \label{3.8}
\bigl\|{\Pi}_N\RR_1^{\hat u}(v_0,\varPhi(h,\hat u_0)v_0)\bigr\|\le C_2\delta\|v_0\|,\quad 
\|\varPhi(h,\hat u_0)\|_\LL\le C_2,
\end{equation}
where $C_2>0$ is a constant not depending on~$N$ and~$\delta$. In this case, the Poincaré inequality and the regularising property of~\eqref{3.6} imply that
\begin{align*}
\|\RR_1^{\hat u}(v_0,\varPhi(h,\hat u_0) v_0)\|
&=\|(I-\Pi_N)\RR_1^{\hat u}(v_0,\varPhi (h,\hat u_0)v_0)\|+C_2\delta\,\|v_0\|\\
&\le \alpha_{N+1}^{-1/2}\|\RR_1^{\hat u}(v_0,\varPhi(h,\hat u_0) v_0)\|_1
+C_2\delta\,\|v_0\|\\
&\le C_3\alpha_{N+1}^{-1/2}\bigl(\|v_0\|
+\|\varPhi (h,\hat u_0)v_0\|_{L^2(D_1)}\bigr)
+C_2\delta\,\|v_0\|\\
&\le \bigl(C_3(C_2+1)\alpha_{N+1}^{-1/2}+C_2\delta\bigr)\|v_0\|.
\end{align*}
Choosing~$N$ sufficiently large and $\delta$ sufficiently small, we obtain the first inequality in~\eqref{3.7}. The second is an immediate consequence of the continuity of~$\RR^{\hat u}$ and the boundedness of~$\varPhi$.  

\smallskip
{\it Step~3: Minimisation problem}. 
The construction of~$\varPhi$ is based on a study of a minimisation problem for solutions of~\eqref{3.6} with a cost functional penalising the term $\|\Pi_Nw(1)\|$. Namely, let us consider the following problem. 

\begin{problem} \label{p3.2}
Given a constant~$\delta>0$, an integer $N\ge1$, and functions $v_0\in H$ and $\hat u\in \XX_1$ satisfying~\eqref{reg}, minimise the functional
$$
J(w,\zeta)=\frac12\int_0^1\|\zeta(t)\|^2dt+\frac1\delta\|\Pi_Nw(1)\|^2
$$
over the set of functions $(w,\zeta)\in\XX_1\times L^2(D_1,\R^2)$ satisfying the equations
\begin{equation} \label{3.9}
\dot w+\langle \hat u,\nabla\rangle w
+\langle w,\nabla\rangle \hat u-\nu \Delta w+\nabla p=\chi({\mathsf P}_m\zeta), \quad \diver w=0, \quad w(0)=v_0,
\end{equation}
where $p=p(t,x)$ is a distribution in~$D_1$ and ${\mathsf P}_m$ stands for the orthogonal projection in~$L^2(D_1)$ onto~$\EE_m$. 
\end{problem}
We shall show that Problem~\ref{p3.2} has a unique solution $(w,\zeta)$, which satisfies the inequality
\begin{equation} \label{3.10}
\frac1\delta\|\Pi_Nw(1)\|^2+\|\zeta\|_{L^2(D_1)}^2\le C_3\|v_0\|^2,
\end{equation}
where $C_3>0$ is a constant not depending on~$N$ and~$\delta$. This will imply the required inequalities~\eqref{3.8}, in which $\varPhi v_0=\chi({\mathsf P}_m\zeta)$. Further analysis shows that the mapping $\hat u\mapsto \zeta$ is smooth from~$\XX_1$ to~$L^2(D_1)$ and uniformly Lipschitz continuous on bounded balls. Since~$\hat u$ is an analytic function of~$(h,\hat u_0)$, this will complete the proof of the theorem. 

\subsection{Minimisation problem}
{\it Step 1: Existence, uniqueness, and linearity.}
Let us prove that Problem~\ref{p3.2} has a unique optimal  solution $(w,\zeta)$ in the space $X:=\XX_1\times L^2(D_1)$. Indeed, the function~$J:X\to\R$ is non-negative and therefore has an infimum on any affine subspace of~$X$. Denote by~$X_{v_0}$ the affine subspace of~$X$ defined by~\eqref{3.9} and by~$J_*$ the infimum of~$J$ on~$X_{v_0}$. Let~$(w_n,\zeta_n)\in X_{v_0}$ be an arbitrary minimising sequence. Then $\{\zeta_n\}$ is a bounded sequence in~$L^2(D_1)$, and without loss of generality we can assume that it converges weakly to a limit~$\zeta$. It follows that the sequence of solutions~$w_n\in\XX_1$ of problem~\eqref{3.9} with $\zeta=\zeta_n$ converges in the space~$\XX_1$ to a limit~$w$, which satisfies~\eqref{3.9}. The lower semi-continuity of the norm of a Hilbert space now implies that 
$$
J(w,\zeta)\le \liminf_{n\to\infty}J(w_n,\zeta_n)=J_*.
$$
Recalling the definition of~$J_*$, we conclude that $J(w,\zeta)=J_*$. 

\smallskip
To prove the uniqueness, note that any affine subspace is a convex set in~$X$. Combining this property with the strict convexity of the norm of a Hilbert space, we see that if $(w_i,\zeta_i)$, $i=1,2$, are two optimal solutions, then $\zeta_1=\zeta_2$. Since the solution of problem~\eqref{3.9} with a given~$\zeta\in L^2(D_1)$ is unique, we conclude that $w_1=w_2$. 

\smallskip
Finally, it is a standard fact of the optimisation theory that the unique minimum of a quadratic functional under a linear constraint can be expressed as a linear function of the problem data. In the case under study, the corresponding operator depends on the reference solution~$\hat u$. We shall denote by $\varPsi(\hat u)$ the linear operator that takes~$v_0$ to~$\zeta$, where $(w,\zeta)$ is the optimal solution for Problem~\ref{p3.2}.  

\smallskip
{\it Step 2: Regularity and Lipschitz continuity.}
We now prove that the mapping~$\varPsi$  regarded as an application from~$H$ to~$\LL(H,L^2(D_1))$ is infinitely differentiable and uniformly Lipschitz continuous on balls. Let us denote by $w_0=w_0(\hat u,v_0)\in\XX_1$ the solution of problem~\eqref{3.9} with $\zeta\equiv0$. The linear constraint of Problem~\ref{p3.2} is equivalent to the relation 
\begin{equation} \label{3.15}
w=w_0+\RR^{\hat u}(0,\chi({\mathsf P}_m\zeta)).
\end{equation}
Setting $A(\hat u)\zeta=\RR_1^{\hat u}(0,\chi({\mathsf P}_m\zeta))$, we see that~$(w,\zeta)\in X$ is a solution of Problem~\ref{p3.2} if and only if~$w$ is the global minimum of the function $F:L^2(D_1)\to\R$ defined by 
$$
F(\zeta;\hat u,v_0)=\frac12\int_0^1\|\zeta(t)\|^2dt
+\frac1\delta\,\bigl\|\Pi_N\bigl(w_0+A(\hat u)\zeta\bigl)\bigr\|^2. 
$$
Using standard methods of the theory of 2D Navier--Stokes equations, we can prove that  $A(\hat u)$ is an analytic function from~$\XX_1$ to~$\LL(H)$
which is uniformly Lipschitz continuous on bounded sets; e.g., see~\cite{kuksin-1982} or Chapter 1 in~\cite{VF}. It follows that~$F$ satisfies the hypotheses of Proposition~\ref{p5.6} in which $U=L^2(D_1)$ and $Y=\XX_1\times H$. Thus, the unique minimum~$\zeta$ of~$F$ is a smooth function of~$(\hat u,v_0)$ valued in~$L^2(D_1)$, and it is Lipschitz continuous on bounded subsets. Recalling relation~\eqref{3.15}, we conclude that the unique solution~$(w,\zeta)\in X$ of Problem~\ref{p3.2} is a smooth function on~$\XX_1\times H$ which is Lipschitz continuous on bounded subsets. Since~$(w,\zeta)$ linearly depends on~$v_0$, it is straightforward  to derive the required properties of~$\varPsi$. 

\smallskip
{\it Step 3: A priori estimate.}
From now on, we fix $R,\delta>0$ and assume that $\hat u\in B_{\YY_{\delta,1}}(R)$. We claim that there is an integer $m\ge1$ depending on~$R$ and~$N$, and a constant~$C>0$ depending only on~$R$ such that, if $\hat u\in B_{\YY_{\delta,1}}(R)$, then inequality~\eqref{3.10} holds. Indeed, note that the constraint given by~\eqref{3.9} is equivalent to the equations
\begin{equation} \label{3.11}
\dot w+\nu Lw+B(\hat u,w)+B(w,\hat u)=\Pi\bigl(\chi({\mathsf P}_m\zeta)\bigr),
\quad w(0)=v_0,
\end{equation}
where we set $B(u_1,u_2)=\Pi(\langle u_1,\nabla\rangle u_2)$. Thus, the pair $(w,\zeta)\in X$ constructed in Step~1 is the minimiser of~$J$ under constraint~\eqref{3.10}. Applying the Kuhn--Tucker theorem (see Chapter~I in~\cite{IT1979}), we can find functions $\lambda\in H$ and $\theta\in L^2(J_1,V)$ such that, for any $(r,\xi)\in X$, we have  
\begin{multline*}
\int_0^1(\zeta,\xi)\,dt+\frac2\delta\bigl(\Pi_Nw(1),r(1)\bigr)
+\bigl(\lambda,r(0)\bigr)\\
+\int_0^1\bigl(\theta,\dot r+\nu Lr+B(\hat u,r)+B(r,\hat u)-\chi({\mathsf P}_m\xi)\bigr)dt
=0.
\end{multline*}
It follows that 
\begin{gather}
\dot \theta-\nu L\theta-B^*(\hat u)\theta=0,\label{3.12}\\
\theta(1)=-\frac2\delta \Pi_Nw(1), \quad \zeta={\mathsf P}_m(\chi \theta),\label{3.13}
\end{gather}
where $B^*(\hat u)$ denotes the (formal) adjoint of the operator $B(\hat u,\cdot)+B(\cdot,\hat u)$ in~$H$, and Eq.~\eqref{3.12} holds in the sense of distributions. Multiplying Eq.~\eqref{3.12} by~$w$ and the first equation in~\eqref{3.11} by~$\theta$, adding together the resulting relations, and integrating over~$D_1$, we derive
\begin{equation} \label{3.14}
\frac2\delta \|\Pi_Nw(1)\|^2+\int_0^1\|\zeta(t)\|^2dt=-(\theta(0),v_0),
\end{equation}
where we used relations~\eqref{3.13} and the initial condition in~\eqref{3.11}. Now note that~\eqref{3.12} is equivalent to the backward Navier--Stokes equations~\eqref{5.12}. Therefore, by Proposition~\ref{p5.5}, the function~$\theta$ must satisfy the truncated observability inequality~\eqref{5.14}. Combining it with~\eqref{3.14}, we obtain~\eqref{3.10}. 
We have thus shown that (cf.~\eqref{3.8})
\begin{equation} \label{3.18}
\bigl\|\Pi_N\RR_1^{\hat u}
\bigl(v_0,\chi{\mathsf P}_m(\varPsi(\hat u)v_0)\bigr)\bigr\|\le C\delta \|v_0\|, 
\quad \|\varPsi(\hat u)v_0\|_{L^2(D_1)}\le C\|v_0\|. 
\end{equation}

\subsection{Proof of Theorem~\ref{t3.1}}
Let us define $\varPhi(h,\hat u_0)$ as the linear operator taking~$v_0$ to~$\chi({\mathsf P}_m\varPsi(\hat u)v_0)$, where $\hat u\in\XX_1$ is the solution of problem~\eqref{1.1}, \eqref{1.5} with $f=h$ and $u_0=\hat u_0$. Standard regularity results for 2D Navier--Stokes equations imply that $\hat u\in\YY_{\delta,1}$ for any $\delta\in(0,1)$ (see Theorem~3.5 in Chapter~III of~\cite{Te}), so that~$\varPsi(\hat u)$ is well defined. Moreover, the norm $\|\hat u\|_{\YY_{\delta,1}}$ is bounded by a constant depending only on~$R$ and~$\delta$. 
We claim that~$\varPhi(h,\hat u_0)$ satisfies the contraction property stated in the theorem. If this assertion is proved, then the regularity and Lipschitz continuity of~$\varPsi$ combined with similar properties  of the resolving operator for the 2D Navier--Stokes system will imply the remaining assertions on~$\varPhi$. 

\smallskip
Inequalities~\eqref{3.18} imply that the solution~$(w,\zeta)$ of Problem~\ref{p3.2} satisfies~\eqref{3.10}. Therefore, choosing~$N$ and $\delta^{-1}$ sufficiently large, we ensure that the function~$w=\RR^{\hat u}(v_0,\varPhi(h,\hat u_0)v_0)$ satisfies~\eqref{3.7}.  Let us represent a solution of the non-linear problem~\eqref{3.4}, \eqref{3.5} with the right-hand side $\eta=\varPhi(h,\hat u_0)(u_0-\hat u_0)$ in the form $v=w+z$. Then $z\in\XX_1$ must be a solution of the problem
\begin{gather}
\dot z+\langle z,\nabla\rangle z+\langle \hat u+w,\nabla\rangle z
+\langle z,\nabla\rangle (\hat u+w)-\nu \Delta z+\nabla p
=-\langle w,\nabla\rangle w,
\label{3.16}\\
\diver z=0, \quad z\bigr|_{\p D}=0, \quad z(0)=0. \label{3.17}
\end{gather}
Taking the scalar product of~\eqref{3.16} with~$2z$ in~$H$ and carrying out some standard transformations, we derive
$$
\p_t\|z\|^2+2\nu\|\nabla z\|^2
\le C_1\bigl(\|\hat u+w\|_1\|z\|+\|w\|_1\|w\|\bigr)\|\nabla z\|. 
$$
It follows that 
$$
\p_t\|z\|^2+\nu\|\nabla z\|^2
\le C_2\bigl(\|\hat u+w\|_1^2\|z\|^2+\|w\|_1^2\|w\|^2\bigr). 
$$
Applying the Gronwall inequality and using the initial condition in~\eqref{3.17}, we obtain
\begin{equation*} 
\|z(t)\|^2\le \int_0^t\exp
\biggl(C_2\int_s^t\|\hat u+w\|_1^2dr\biggr)\|w\|_1^2\|w\|^2ds. 
\end{equation*}
Recalling that $\|\hat u\|_{\XX_1}$ is bounded by a constant depending only on~$R$ and using the second inequality in~\eqref{3.7}, we derive
$$
\sup_{0\le t\le 1}\|z(1)\|\le C_3(R)\|w\|_{\XX_1}^2
\le C_4(R)\|u_0-\hat u_0\|^2.
$$
Since $\|u_0-\hat u_0\|\le d$, choosing~$d>0$ sufficiently small ensures that 
$$
\|z(1)\|\le\frac{q}{2}\|u_0-\hat u_0\|.
$$
Combining this with the first inequality of~\eqref{3.7} in which $v_0=u_0-\hat u_0$, we get~\eqref{3.2}. This completes the proof of the theorem. 

\section{Proof of Theorem~\ref{t2.1}}
\label{s4}

We first outline the main steps. A well-known dissipativity argument shows that the random dynamical system defined by~\eqref{2.3} has a compact invariant absorbing set~$X\subset H$. Therefore it suffices to prove the uniqueness of an invariant measure and the property of exponential mixing for the restriction of~$(u_k,\IP_u)$ to~$X$.  This will be done with the help of Proposition~\ref{p2.2}. Namely, we shall prove Proposition~\ref{p2.3} and use relations~\eqref{2.27} and~\eqref{2.28} to define a Markov chain $(\uuu_k,\IP_\uuu)$ in the extended phase space $\XXX=X\times X$. This Markov chain is an extension of~$(u_k,\IP_u)$ and possesses the recurrence and exponential squeezing properties of Section~\ref{s2.2}, and therefore the hypotheses of Proposition~\ref{p2.2} are satisfied. This will complete the proof of Theorem~\ref{t2.1}. 

\subsection{Reduction to a compact phase space}
\label{s4.1}
Well-known properties of the resolving operator for the Navier--Stokes system imply that~$S$ satisfies the inequality
$$
\|S(u,f)\|\le \varkappa\,\|u\|+C_1\|f\|_{L^2(D_1)}\quad
\mbox{for $u\in B_H(R)$, $f\in L^2(D_1)$}
$$
where $\varkappa<1$ and $C_1>0$ are some universal constants. Let $r>0$ be so large $\|h+\eta_k\|_{L^2(D_1)}\le r$ almost surely. Then, with probability~$1$, we have
\begin{equation} \label{4.1}
\|S(u,h+\eta_k)\|\le \varkappa \rho+C_1r\quad\mbox{for $u\in B_H(\rho)$}. 
\end{equation}
It follows that if $R\ge\frac{C_1r}{1-\varkappa}$, then the ball $B_H(R)$ is invariant for the Markov chain $(u_k,\IP_u)$. Let us take $R=\frac{2C_1r}{1-\varkappa}$ and denote by~$X$ the image of the set $B_H(R)\times B_{L^2(D_1)}(r)$ under the mapping~$S$. Then~$X$ is an invariant subset for $(u_k,\IP_u)$, and the regularising property of the Navier--Stokes dynamics implies that~$X$ is compact in~$H$. Iterating~\eqref{4.1}, we see that 
$$
\IP_u\{u_k\in X\mbox{ for }k\ge k_0(\rho)\}=1\quad
\mbox{for any $u\in B_H(\rho)$},
$$
where $k_0(\rho)=(\ln\rho+C_2)/\ln\varkappa^{-1}$ with a large constant~$C_2>0$. It follows that~$(u_k,\IP_u)$ has at least one stationary measure~$\mu$, and any such measure is supported by~$X$. It is easy to see that to prove~\eqref{2.10}, it suffices to establish inequality~\eqref{2.26}. The latter is proved in the next three subsections.

\subsection{Proof of Proposition~\ref{p2.3}}
\label{s4.2}
We shall apply Proposition~\ref{p5.4} to construct a measurable coupling~$(V,V')$ and Proposition~\ref{p5.3} to prove~\eqref{2.27}. Namely, fix $R>0$ so large that $X\subset B_H(R-1)$ and $\|h+\eta_k\|_{H^1(D_1)}\le R-1$ almost surely. Denote by~$C, d>0$ and~$m\ge1$ the parameters constructed in Theorem~\ref{t3.1} with $q=1/4$, define the Polish space
$$
Z=\{(u,u')\in X\times X:\|u-u'\|\le d\},
$$
and introduce the function $\e(u,u')=\frac12\|u-u'\|$ on the space~$Z$. Let us consider the pair of measures $(P_1(u,\cdot),P_1(u',\cdot))$. By Proposition~\ref{p5.4} with $\theta=\frac12$, there is a probability space~$(\Omega,\FF,\IP)$ and measurable functions $V,V':\Omega\times Z\to H$ such that for any $(u,u')\in Z$ the pair $(V(u,u'),V'(u,u'))$ is a coupling for $(P_1(u,\cdot),P_1(u',\cdot))$ and 
\begin{equation} \label{4.2}
\IP\bigl\{\|V(u,u')-V'(u,u')\|>\tfrac12\|u-u'\|\bigr\}
\le C_{\|u-u'\|/4}\bigl(P_1(u,\cdot),P_1(u',\cdot)\bigr). 
\end{equation}
In view of~\eqref{5.5}, to estimate the right-hand side of this inequality, it suffices to bound the function $K_{\|u-u'\|/4}(P_1(u,\cdot),P_1(u',\cdot))$. 

\smallskip
To this end, we shall apply Propositions~\ref{p5.3}, \ref{p5.7} and Theorem~\ref{t3.1}. Let us endow the ball $B_R\subset H^1(D_1)$ with the law~$\lambda$ of the random variables~$\eta_k$. Then $P_1(u,\cdot)$ is the image of~$\lambda$ under the mapping $\eta\mapsto S(u,h+\eta)$ acting from~$B_R$ to~$H$. Let~$\EE_m$ be the subspace entering Theorem~\ref{t3.1} and let 
$$
\varPhi:B_R\times B_H(R)\to \LL(H,\EE_m)
$$
be a smooth $C$-Lipschitz mapping such that
\begin{equation} \label{4.3}
\bigl\|S\bigl(u,h+\eta\bigr)
-S\bigl(u',h+\eta+\varPhi(h+\eta,u)(u-u')\bigr)\bigr\|\le\frac14\|u-u'\|
\end{equation}
for  any $\eta\in B_R$ and any $u,u'\in B_H(R)$ satisfying the inequality $\|u-u'\|\le d$. The existence of such mapping was established in Theorem~\ref{t3.1}. Let us define a transformation $\varPsi=\varPsi_{u,u'}$ of the space~$H^1(D_1)$ by the relation
$$
\varPsi(\eta)=\eta+\chi\bigl(\|h+\eta\|_1\bigr)\varPhi(h+\eta,u)(u-u'), 
$$
where $\chi:\R_+\to\R_+$ is a smooth function such that $\chi(r)=1$ for $r\le R-1$ and $\chi(r)=0$ for $r\ge R$. The choice of the constant~$R$ and inequality~\eqref{4.3} imply that 
$$
\bigl\|S\bigl(u,h+\eta\bigr)
-S\bigl(u',h+\varPsi_{u,u'}(\eta)\bigr)\bigr\|\le\frac14\|u-u'\|
\quad\mbox{for $\lambda$-a.e. $\eta\in H^1(D_1)$}. 
$$
Therefore, by Proposition~\ref{p5.3}, we have
\begin{equation} \label{4.4}
K_{\|u-u'\|/4}(P_1(u,\cdot),P_1(u',\cdot))
\le 2\,\|\lambda-\varPsi_*(\lambda)\|_{\mathrm{var}}. 
\end{equation}
Now note that the mapping~$\varPsi$ satisfies the hypotheses  of Proposition~\ref{p5.7} with a constant~$\varkappa$ proportional to $\|u-u'\|$.  Combining~\eqref{4.2}, \eqref{4.4}, and~\eqref{5.22}, we arrive at the required inequality~\eqref{2.27}.

\subsection{Recurrence}
\label{s4.3}
Let us recall that the probability space $(\OOmega,\FFF,\IP)$ and the $\XXX$-valued  Markov chain~$\{\uuu_k=(u_k,u_k')\}$ on it were defined in Section~\ref{s2.2}. We shall denote by~$\IP_\uuu$ the probability measure associated with the initial condition~$\uuu$. 
Let us set 
$$
\BBB=\{\uuu=(u,u')\in\XXX:\|u-u'\|\le d\},
$$
where $d>0$ is a small constant to be chosen later. In this and next subsections, we shall prove that the recurrence and exponential squeezing properties are satisfied for~$(\uuu_k,\IP_\uuu)$ with the above choice of~$\BBB$. 

\smallskip
To prove that $\ttau(\BBB)$ is $\IP_\uuu$-almost surely finite and satisfies~\eqref{2.22}, it suffices to show that 
\begin{equation} \label{4.21}
p:=\sup_{\uuu\in\XXX}\IP_\uuu\{\ttau(\BBB)>\ell\}<1,
\end{equation}
where $\ell\ge1$ is an integer. Indeed, if this inequality is proved, then using the Markov property, for any integer $m\ge1$ we can write
\begin{align*}
\IP_\uuu\{\ttau(\BBB)>m\ell\}
&=\E_\uuu\IP_\uuu\{\ttau(\BBB)>m\ell\,|\,\FFF_{(m-1)\ell}\}\notag\\
&=\E_\uuu\bigl(I_{\ttau(\BBB)>(m-1)\ell}\,\IP_{\uuu_{(m-1)\ell}}\{\ttau(\BBB)>\ell\}\bigr)\notag\\
&\le p\,\IP_\uuu\{\ttau(\BBB)>(m-1)\ell\}, 
\end{align*}
where $\{\FFF_k\}$ stands for the filtration generated by the Markov family $(\uuu_k,\IP_\uuu)$. By iteration, the above inequality implies that
\begin{equation} \label{4.22}
\IP_\uuu\{\ttau(\BBB)>m\ell\}\le p^m\quad\mbox{for all $m\ge1$}. 
\end{equation}
A simple application of the Borel--Cantelli lemma now implies~$\ttau(\BBB)$ is $\IP_\uuu$-almost surely finite. Furthermore, inequality~\eqref{4.22} immediately implies that~$\ttau(\BBB)$ satisfies~\eqref{2.22}. 

We now prove~\eqref{4.21}. Let~$\eta$ and~$\eta'$ be the random variables entering the definition of the coupling operator~$\RRR$ (see~\eqref{2.28}) and let
$$
\zeta_k(\omega)=\eta(\omega_k), \quad \zeta_k'(\omega)=\eta'(\omega_k), 
\quad k\ge1.
$$
Then $\{\zeta_k,\zeta_k',k\ge1\}$ is a family of i.i.d.\ random variables in~$L^2(D_1)$ whose law coincides with that of~$\eta_k$. 
Let $\hat u\in H$ be the point defined in Hypothesis~(H2). Since~$X$ is a closed absorbing subset, we have $\hat u\in X$. Using the definitions of~$\RR$ and~$\ttau$, we can write
\begin{align*}
\IP_\uuu\{\ttau(\BBB)>\ell\}&=\IP_\uuu\{\ttau(\BBB)>\ell, \|u_\ell-u_\ell'\|>d\}\\
&=\IP_\uuu\{\ttau(\BBB)>\ell, \|S_\ell(u,\zeta_1,\dots,\zeta_\ell)-S_\ell(u,\zeta_1',\dots,\zeta_\ell')\|>d\}\\
&\le \IP\{\|S_\ell(u,\zeta_1,\dots,\zeta_\ell)-S_\ell(u',\zeta_1',\dots,\zeta_\ell')\|>d\}\\
&= 1-\IP\{\|S_\ell(u,\zeta_1,\dots,\zeta_\ell)-S_\ell(u',\zeta_1',\dots,\zeta_\ell')\|\le d\}\\
&\le 1-P(\ell,d)^2,
\end{align*}
where the operator~$S_l$ is defined in Hypothesis~(H2), and we set 
$$
P(\ell,d)=\inf_{v\in X}\IP\{\|S_\ell(v,\zeta_1,\dots,\zeta_\ell)-\hat u\|\le d/2\}.
$$
Thus, inequality~\eqref{4.21} will be established if we show that $P(\ell,d)>0$ for a sufficiently large integer~$\ell$. 

To this end, let us denote by~$\ell$ the integer defined in Hypothesis~(H2) with $\e=d/4$ and~$R>0$ so large that $X\subset B_H(R)$. We fix $v\in X$ and denote by $\zeta_k=\zeta_k^v\in\KK$ some vectors for which~\eqref{2.8} holds. Let $\delta>0$ be so small that for any $\tilde\zeta_k\in\KK$, $k=1,\dots,\ell$, satisfying the inequalities
$\|\tilde\zeta_k-\zeta_k^v\|<\delta$ we have 
$$
\|S_\ell(v,\tilde\zeta_1,\dots,\tilde\zeta_\ell)-\hat u\|\le d/2. 
$$
Define the event
$\Gamma(v,\ell)=\{\|\zeta_k-\zeta_k^v\|<\delta,k=1,\dots,\ell\}\subset\OOmega$. What has been said implies that
$$
\|S_\ell(v,\zeta_1,\dots,\zeta_\ell)-\hat u\|\le d/2\quad
\mbox{for $\oomega\in\Gamma(v,\ell)$}. 
$$
We see that 
\begin{align}
P(\ell,d)&\ge\inf_{v\in X}\IP\bigl(\Gamma(v,\ell)\bigr)
\ge\inf_{v\in X}\prod_{k=1}^\ell\IP\{\|\zeta_k-\zeta_k^v\|<\delta\}\notag\\
&\ge \Bigl(\,\inf_{\xi\in\KK}\IP\{\|\zeta_1-\xi\|<\delta\}\Bigr)^\ell, \label{4.23}
\end{align}
where we used the fact that~$\zeta_k$ are i.i.d.\ random variables. It remains to note that the function $\xi\mapsto\IP\{\|\zeta_1-\xi\|<\delta\}$ is lower semicontinuous and strictly positive on~$\KK$, and therefore the right-hand side of~\eqref{4.23} is a positive constant. This completes the verification of the recurrence property. 

\subsection{Exponential squeezing}
\label{s4.4} 
Let us take arbitrary initial points $u,u'\in X$ such that $(u,u')\in\BBB$. Then, by~\eqref{2.27} and~\eqref{2.28}, we have
\begin{equation} \label{4.24}
\IP_\uuu\bigl\{\|u_1-u_1'\|\le\tfrac12\|u-u'\|\bigr\}\ge 1-C\,\|u-u'\|. 
\end{equation}
For any integer $n\ge1$, let us set 
$$
\Gamma_n=\{\|u_k-u_k'\|\le\tfrac12\|u_{k-1}-u_{k-1}'\|\mbox{ for }1\le k\le n\}.
$$ 
Note that $\|u_k-u_k'\|\le 2^{-k}\|u_0-u_0'\|$ for $1\le k\le n$ on the set $\Gamma_n$.
Combining this observation with inequality~\eqref{4.24} and the Markov property,  we derive
\begin{align*}
\IP_\uuu(\Gamma_n)
&=\E_\uuu\bigl(I_{\Gamma_{n-1}}
\IP_\uuu\bigl\{\|u_n-u_n'\|\le\tfrac12\|u_{n-1}-u_{n-1}'\|\,|\,\FFF_{n-1}\bigr\}\bigr)\\
&=\E_\uuu\bigl(I_{\Gamma_{n-1}}\IP_{\uuu_{n-1}}\bigl\{\|u_1-u_1'\|\le\tfrac12\|u_0-u_0'\|\bigr\}\bigr)\\
&\ge \bigl(1-2^{1-n}C\|u-u'\|\bigr)\,\IP_\uuu(\Gamma_{n-1}). 
\end{align*}
Iterating this inequality, for any $n\ge1$ we obtain
\begin{equation} \label{4.25}
\IP_\uuu(\Gamma_n)\ge\prod_{k=1}^n\bigl(1-2^{1-k}C\|u-u'\|\bigr)
\ge 1-C_1\|u-u'\|.
\end{equation}

Let us define the Markov time~$\sigma$ by relation~\eqref{2.23} with $C=d$ and $\beta=\ln2$. Then, for any $\uuu\in\BBB$, with $\IP_\uuu$-probability~$1$ we have
\begin{align*}
\{\sigma=\infty\}&=\{\|u_k-u_k'\|\le2^{-k}d\mbox{ for }k\ge1\}\\
&\supset\{\|u_k-u_k'\|\le\tfrac12\|u_{k-1}-u_{k-1}'\|\mbox{ for }k\ge1\}
=\bigcap_{n=1}^\infty \Gamma_n,
\end{align*}
whence it follows that 
$$
\IP_\uuu\{\sigma=\infty\}=\lim_{n\to\infty}\IP_\uuu(\Gamma_n)
\ge1-C_1\|u-u'\|.
$$
Choosing $d>0$ sufficiently small, we arrive at~\eqref{2.24}.  

To prove inequality~\eqref{2.25}, it suffices to show that 
\begin{equation} \label{4.26}
\IP_\uuu\{\sigma=n\}\le C_22^{-n}\quad
\mbox{for $n\ge1$, $\uuu\in\BBB$},
\end{equation}
where the positive constant~$C$ does not depend on~$\uuu$. To this end, note that
$$
\{\sigma=n\}=\{\sigma>n-1\}\cap\{\|u_n-u_n'\|>2^{-n}d\}.
$$
Using the Markov property, inequality~\eqref{4.24}, and the fact that $\|u_k-u_k'\|\le 2^{-k}d$ on the set $\{\sigma>k\}$, we derive
\begin{align*}
\IP_\uuu\{\sigma=n\}
&=\E_\uuu\bigl(I_{\{\sigma>n-1\}}\IP_\uuu\{\|u_n-u_n'\|>2^{-n}d\,|\,\FFF_{n-1}\}\bigr)\\
&=\E_\uuu\bigl(I_{\{\sigma>n-1\}}\IP_{\uuu_{n-1}}\{\|u_1-u_1'\|>2^{-n}d\}\bigr)\\
&\le\E_\uuu\bigl(I_{\{\sigma>n-1\}}\IP_{\uuu_{n-1}}\{\|u_1-u_1'\|>\tfrac12\|u_0-u_0'\|\}\bigr)\\
&\le 2^{1-n}Cd\,\IP_\uuu\{\sigma>n-1\}. 
\end{align*}
We thus obtain inequality~\eqref{4.26} with $C_2=Cd$. The proof of Theorem~\ref{t2.1} is complete.

\section{Appendix}
\label{s5}
\subsection{Optimal coupling}
\label{s5.1}
Let~$X$ be a Polish space with a metric~$d$ and let $\mu_1,\mu_2\in\PP(X)$. Recall that  a pair of $X$-valued random variables~$(\xi_1,\xi_2)$ is called a {\it coupling\/} for~$(\mu_1,\mu_2)$ if $\DD(\xi_i)=\mu_i$, $i=1,2$. We denote by~$\Pi(\mu_1,\mu_2)$ the set of all couplings for~$(\mu_1,\mu_2)$. Let us fix~$\e\ge0$ and define a symmetric function $d_\e:X\times X\to\R$ by the relation
\begin{equation} \label{5.1}
d_\e(u_1,u_2)=\left\{
\begin{aligned}
1 & \quad\mbox{if $d(u_1,u_2)>\e$},\\
0 & \quad\mbox{if $d(u_1,u_2)\le\e$}.
\end{aligned}
\right.
\end{equation}

\begin{definition}
We shall say that a coupling $(\xi_1,\xi_2)\in\Pi(\mu_1,\mu_2)$ is {\it $\e$-optimal\/} if it minimizes the function $(\zeta_1,\zeta_2)\mapsto \E\,d_\e(\zeta_1,\zeta_2)$ defined on the set of all couplings for~$(\mu_1,\mu_2)$. That is, 
\begin{equation} \label{5.2}
\E\,d_\e(\xi_1,\xi_2)=\inf_{(\zeta_1,\zeta_2)\in\Pi(\mu_1,\mu_2)}
\E\,d_\e(\zeta_1,\zeta_2).
\end{equation}
In particular, for~$\e=0$ we obtain the usual concept of maximal coupling of measures; e.g., see~\cite{thorisson2000}. 
\end{definition}

Let us denote by~$C_\e(\mu_1,\mu_2)$ the quantity defined by the right-hand side of~\eqref{5.2} and call it the {\it $\e$-optimal cost\/} for the pair~$(\mu_1,\mu_2)$. Note that, in view of~\eqref{5.1} and~\eqref{5.2}, for any $\e$-optimal coupling~$(\xi_1,\xi_2)$ we have
\begin{equation} \label{5.3}
\IP\{d(\xi_1,\xi_2)>\e\}=C_\e(\mu_1,\mu_2).
\end{equation}
Thus, to estimate the probability of the event that the distance between the components of an $\e$-optimal coupling for~$(\mu_1,\mu_2)$ is larger than~$\e$ it suffices to estimate the corresponding $\e$-optimal cost. We now establish a simple result that enables one to do it. Let us introduce the following function on the space $\PP(X)\times\PP(X)$:
\begin{equation} \label{5.4}
K_\e(\mu_1,\mu_2)
=\sup_{f,g}\bigl((f,\mu_1)-(g,\mu_2)\bigr),
\end{equation}
where the supremum is taken over all functions $f,g\in C_b(X)$ satisfying the inequality 
\begin{equation} \label{5.10}
f(u_1)-g(u_2)\le d_\e(u_1,u_2)\quad\mbox{for $u_1,u_2\in X$}.
\end{equation}
Then, by the Kantorovich duality (see Theorem~5.10 in~\cite{villani2009}), we have 
\begin{equation} \label{5.5}
K_\e(\mu_1,\mu_2)=C_\e(\mu_1,\mu_2)
\end{equation}
Thus, to estimate the $\e$-optimal cost it suffices to estimate the function~$K_\e$. The following proposition reduces this question to a ``control'' problem. 

\begin{proposition} \label{p5.3}
Let $X$ be a compact metric space with metric~$d$, let $U_1,U_2$ be two $X$-valued random variables defined on a probability space $(\Omega,\FF,\IP)$, and let $\mu_1,\mu_2$ be their laws. Suppose there is a measurable mapping $\varPsi:\Omega\to\Omega$ such that 
\begin{equation} \label{5.7}
d\bigl(U_1(\omega),U_2(\varPsi(\omega))\bigr)\le\e
\quad\mbox{for almost every $\omega\in\Omega$},
\end{equation}
where $\e>0$ is a constant. Then 
\begin{equation} \label{5.8}
K_\e(\mu_1,\mu_2)\le 2\,\|\IP-\varPsi_*(\IP)\|_{\rm var}.
\end{equation}
In particular, any $\e$-optimal coupling $(\xi_1,\xi_2)$ for the pair~$(\mu_1,\mu_2)$ satisfies the inequality
\begin{equation} \label{5.9}
\IP\{d(\xi_1,\xi_2)>\e\}\le 2\,\|\IP-\varPsi_*(\IP)\|_{\rm var}.
\end{equation}
\end{proposition}

\begin{proof}
Inequality~\eqref{5.9} is a straightforward consequence of~\eqref{5.8}, \eqref{5.5}, and~\eqref{5.3}.  To prove~\eqref{5.8}, 
we use an argument applied in the proof of Lemma~11.8.6 in~\cite{dudley2002}. Namely, note that if $f,g\in C_b(X)$ are such that~\eqref{5.10} holds,  then the function $h(u)=\sup_{v\in X}(f(v)-d_\e(u,v))$ satisfies the inequalities $h\le g$ and
$$
f(u_1)-h(u_2)\le d_\e(u_1,u_2), \quad |h(u_1)-h(u_2)|\le1
\quad\mbox{for $u_1,u_2\in X$}.
$$
It follows that
\begin{align*}
(f,\mu_1)-(g,\mu_2)
&=\E\,\bigl(f(U_1)-g(U_2)\bigr)\\
&\le\E\,\bigl(f(U_1)-h(U_2\circ\varPsi)\bigr)
+\E\,\bigl(h(U_2\circ\varPsi)-h(U_2)\bigr)\\
&\le \E\,d_\e(U_1,U_2\circ\varPsi)
+(h\circ U_2,\varPsi_*(\IP))-(h\circ U_2,\IP)\\
&\le \bigl|(h\circ U_2,\IP)-(h\circ U_2,\varPsi_*(\IP))\bigr|.
\end{align*}
It remains to take the supremum over $f,g$ and to note that right-hand side of this inequality does not exceed that of~\eqref{5.8}. 
\end{proof}

We now study the question of existence and measurability of an $\e$-optimal coupling. Let~$(Z,\BB)$ be a measurable space and let $\{\mu_i^z, z\in Z\}$, $i=1,2$, be two families of probability measures on~$X$ such that the mapping $z\mapsto \mu_i^z$ is measurable from~$Z$ to the space~$\PP(X)$ endowed with the topology of weak convergence and the corresponding $\sigma$-algebra. The following proposition establishes the existence of an  ``almost'' $\e$-optimal coupling that is a measurable function of~$z$, provided that~$X$ is a subset of a Banach space. 

\begin{proposition} \label{p5.4}
In addition to the above hypotheses, assume that $X$ is a compact subset of a Banach space~$Y$. Then for any positive measurable function $\e(z)$, $z\in Z$, and any $\theta\in(0,1)$ there is a probability space $(\Omega,\FF,\IP)$ and measurable functions $\xi_i^z(\omega):\Omega\times Z\to Y$, $i=1,2$, such that, for any $z\in Z$, the pair $(\xi_1^z,\xi_2^z)$ is a coupling for~$(\mu_1^z,\mu_2^z)$ and 
\begin{equation} \label{5.6}
\E\,d_{\e(z)}(\xi_1^z,\xi_2^z)\le C_{\theta\e(z)}(\mu_1^z,\mu_2^z). 
\end{equation}
\end{proposition}

\begin{proof}
The existence of a measurable optimal $\e(z)$-coupling would be an immediate consequence of Corollary 5.22 in~\cite{villani2009} if the function~$d_{\e(z)}(u_1,u_2)$ was continuous and independent of~$z$. Since this is not the case, we now outline the modifications that are needed to prove the existence of a measurable coupling satisfying~\eqref{5.6}. 

Let us write~$Z$ as the union of countably many disjoint measurable subset~$Z_k$ such that the image of the restriction of~$\e(z)$ to any of them is contained in a bounded interval separated from zero. For instance, we can take $Z_k=\{z\in Z:(k+1)^{-1}<\e(z)\le k^{-1}\}$ with $k\ge1$ and use a similar partition for $(1,+\infty)$. If we construct measurable couplings on each set~$Z_k$ for which~\eqref{5.6} holds with $Z=Z_k$, then by gluing them together, we get the required random variables~$\xi_i^z$, $i=1,2$. 

To construct a measurable coupling on~$Z_k$, we first reduce the problem to the case when~$\e\equiv1$. Let us consider stretched measures~$\tilde \mu_i^z$ defined by the formula $\tilde \mu_i^z(\Gamma)=\mu_i^z(\e(z)\Gamma)$ for $\Gamma\in\BB_Y$. The measures~$\tilde \mu_i^z$ are supported by a compact set $X_k\subset Y$, and if~$(\tilde\xi_1^z,\tilde\xi_2^z)$  is a measurable coupling such that 
$$
\E\,d_1(\tilde\xi_1^z,\tilde\xi_2^z)\le C_\theta(\tilde\mu_1^z,\tilde\mu_2^z)
\quad\mbox{for any $z\in Z_k$},
$$ 
then $\e(z)(\tilde\xi_1^z,\tilde\xi_2^z)$ is a measurable coupling for the original measures that satisfies~\eqref{5.6}. Thus, we can assume from the very beginning that $\e\equiv1$. 

Let $d:X\times X\to\R$ be an arbitrary continuous symmetric function such that 
\begin{equation} \label{5.11}
d_1(u_1,u_2)\le d(u_1,u_2)\le d_\theta(u_1,u_2)\quad\mbox{for all $u_1,u_2\in X$}.
\end{equation}
By Corollary 5.22 in~\cite{villani2009}, there is a probability space $(\Omega,\FF,\IP)$ and measurable functions $\xi_i^z:Z\times\Omega\to X$ such that $(\xi_1^z,\xi_2^z)$ is a coupling for~$(\mu_1^z,\mu_2^z)$ for any $z\in Z$, and
$$
\E\,d(\xi_1^z,\xi_2^z)=\inf_{(\zeta_1,\zeta_2)\in\Pi(\mu_1^z,\mu_2^z)}
\E\,d(\zeta_1,\zeta_2). 
$$
Combining this relation with~\eqref{5.11}, we arrive at~\eqref{5.6} with $\e\equiv1$.
\end{proof}

\subsection{Truncated observability inequality}
\label{s5.2}
Let $\delta\in(0,1)$ and let $\hat u\in\YY_{\delta,1}$ be an arbitrary function. Consider the adjoint equation for the Navier--Stokes system linearised around~$\hat u$:
\begin{equation} \label{5.12}
\dot g+\langle \hat u,\nabla\rangle g+\langle \nabla g,\hat u\rangle +\nu\Delta g+\nabla \pi=0, \quad \diver g=0. 
\end{equation}
Let us fix an open set $Q\subset (\delta,1)\times D$, an orthonormal basis $\{\varphi_j\}\subset L^2(Q,\R^2)$, and a function $\chi\in C_0^\infty(Q)$ and denote by~${\mathsf P}_m$ the orthogonal projection in $L^2(Q,\R^2)$ onto the $m$-dimensional subspace spanned by~$\varphi_j$, $j=1,\dots, m$. Recall that~$H_N$ stands for the vector span of the first~$N$ eigenfunctions of the Stokes operator. 
The following result is a simple consequence of the observability inequality for the linearised Navier--Stokes system (see~\cite{FE-1999,imanuvilov-2001,FGIP-2004}).

\begin{proposition} \label{p5.5}
For any $\delta,\rho>0$ and any integer~$N\ge1$ there is an integer $m\ge1$ such that if $\hat u\in B_{\YY_{\delta,1}}(\rho)$, then any solution $g\in\XX_1$ of Eq.~\eqref{5.12} with $g(1)\in H_N$ satisfies the inequality
\begin{equation} \label{5.14}
\|g(0)\|\le C\,\bigl\|{\mathsf P}_m(\chi g)\bigr\|_{L^2(D_1)},
\end{equation}
where $C>0$ is a constant depending only on~$\delta$ and~$\rho$. 
\end{proposition}

\begin{proof}
We essentially repeat the argument used in the paper~\cite{BRS-2010} in which inequality~\eqref{5.14} is proved for a time-independent function~$\chi$ and a projection~${\mathsf P}_m$ acting in the space variables. We claim that if $g\in\XX_1$ is a solution of~\eqref{5.12} with $g(1)\in H_N$, then 
\begin{equation} \label{5.15}
\bigl\|\chi g\bigr\|_{H^1(D_1)}\le C_1\bigl\|\chi g\bigr\|_{L^2(D_1)},
\end{equation}
where $C_1>0$ is a constant depending on~$\delta$, $\rho$,  and~$N$. 
Once this inequality is established, the required result can be derived by using the fact that
$$
\|(I-{\mathsf P}_m)v\|\le \delta_m\|v\|_1\quad 
\mbox{for any $v\in H_0^1(Q,\R^2)$},
$$
where ~$\{\delta_m\}$ is a sequence going to zero as $m\to\infty$; cf.\ proof of Proposition~5.3 in~\cite{BRS-2010}.

\smallskip
Suppose that~\eqref{5.15} is false. Then there are functions~$\hat u_n\in B_{\YY_{\delta,1}}(\rho)$ and solutions $g_n\in\XX_1$ of~\eqref{5.12} with $\hat u=\hat u_n$ such that 
\begin{gather}
g_n(1)\in H_N, \quad \|g_n(1)\|=1,\label{5.16}\\
\bigl\|\chi g_n\bigr\|_{H^1(D_1)}\ge n\,\bigl\|\chi g_n\bigr\|_{L^2(D_1)}. 
\label{5.17}
\end{gather}
Now note that $\|g_n(1)\|_{H^2}\le C_2$ for all $n\ge1$, whence it follows, by standard estimates for the 2D Navier--Stokes system, that 
$$
\|g_n\|_{\XX_1}+\|g_n\|_{L^2(J,H^3)}+\|\p_tg_n\|_{L^2(J,V)}\le C_3, \quad n\ge1.
$$
Passing to a subsequence, we can assume that $\{\hat u_n\}$ and~$\{g_n\}$ converge weakly (in appropriate functional spaces) to some functions~$\hat u$ and~$g$, respectively, such that~$\hat u\in B_{\YY_{\delta,1}}(\rho)$, $g\in\XX_1$ is a solution of~\eqref{5.12}, and $\|g(1)\|=1$. Moreover, it follows from~\eqref{5.17} that $\chi g\equiv 0$. Let us show that the latter is impossible. 

Indeed, let an interval $(a,b)\subset J$ and a ball $B\subset D$ be such that 
$$
\chi(t,x)\ge \alpha>0\quad\mbox{for $(t,x)\in (a,b)\times B$}.
$$
Then, by the observability inequality (see Lemma~1 in~\cite{FGIP-2004}), we obtain
$$
\|g(a)\|\le C_4\|g\|_{L^2((a,b)\times D)}\le C_4\alpha^{-1}\|\chi g\|_{L^2(D_1)}=0.
$$
The backward uniqueness for the linearised Navier--Stokes system (see Section~II.8 in~\cite{BV1992}) implies that $g(t)=0$ for $a\le t\le 1$. This contradicts the fact that $\|g(1)\|=1$. The proof is complete. 
\end{proof}

\subsection{Minima of second-order polynomials on Hilbert spaces}
\label{s5.3}

Let $U$ and~$Y$ be real separable Hilbert spaces and let $F:U\times Y\to\R$ be a function of the form
$$
F(u,y)=(Q_yu,u)_U+(a_y,u)_U+b_y,
$$
where $Q_y\in\LL(U)$ is a self-adjoint operator, $a_y, b_y\in U$ for all $y\in Y$, and $(\cdot,\cdot)_U$ stands for the scalar product in~$U$. We assume that 
\begin{equation} \label{5.18}
(Q_yu,u)_U\ge c\|u\|_U^2\quad\mbox{for all $u\in U$, $y\in Y$.}
\end{equation}
In this case, it is easy to see that for any $y\in Y$ the function $u\mapsto F(u,y)$ has a unique global minimum $u^*=u^*(y)$. The following simple proposition establishes some properties of~$u^*$. 

\begin{proposition} \label{p5.6} 
Let us assume that the functions 
\begin{equation} \label{5.19}
y\mapsto Q_y, \quad y\mapsto (a_y,b_y)
\end{equation}
are infinitely smooth from~$Y$ to the spaces $\LL(U)$ and $U\times U$, respectively. Then the unique minimum~$u^*(y)$ is an infinitely smooth function of $y\in Y$. Moreover, this implication remains true if the property of infinite smoothness is replaced by Lipschitz continuity on bounded balls. 
\end{proposition}

\begin{proof}
Inequality~\eqref{5.18} implies that $u^*(y)$ is the only solution of the linear equation 
\begin{equation} \label{5.20}
(\p_uF)(u,y)=0. 
\end{equation}
Therefore, the required smoothness of~$u^*$ will be established if we show that the implicit function theorem can be applied to~\eqref{5.20}. This is a straightforward consequence of the fact that~$\p_uF$ is an infinitely smooth function of its arguments, and its derivative in~$u$ coincides with $2Q_y$. 

We now prove that if~$a_y$, $b_y$, and~$Q_y$  are Lipschitz continuous on bounded balls, then so is~$u^*(y)$. Indeed, fix~$R>0$ and take two points $y_1,y_2\in B_Y(R)$. It follows from~\eqref{5.20} and the explicit form of the derivative~$\p_uF$ that
$$
Q_{y_1}\bigl(u^*(y_1)-u^*(y_2)\bigr)=\bigl(Q_{y_2}-Q_{y_1}\bigr)u^*(y_2)
+\frac12\bigl(a_{y_2}-a_{y_1}\bigr). 
$$
Taking the scalar product of this equation with $u^*(y_1)-u^*(y_2)$ and using inequality~\eqref{5.18} and the Lipschitz continuity~$Q_y$ and~$a_y$, we obtain the required result. 
\end{proof}

\subsection{Image of measures under finite-dimensional transformations}
\label{s5.4}
Let $X$ be a separable Banach space  endowed with a norm~$\|\cdot\|$ and  represented as the direct sum 
\begin{equation} \label{5.21}
X=E\dotplus F,
\end{equation}
where $E\subset X$ is a finite-dimensional subspace. Denote by~$\mathsf P_E$ and~$\mathsf P_F$ the projections corresponding to decomposition~\eqref{5.21}. Let~$\lambda\in\PP(X)$ be a measure that can be written as the tensor product of its projections $\lambda_E=(\mathsf P_E)_*\lambda$ and $\lambda_F=(\mathsf P_F)_*\lambda$. Assume also that~$\lambda$ has a bounded support  and that~$\lambda_E$ possesses a $C^1$-smooth density~$\rho$ with respect to the Lebesgue measure on~$E$. The following simple result gives an estimate for the total variation distance between the measure~$\lambda$ and its image under a diffeomorphism of~$X$ acting only along~$E$. 

\begin{proposition} \label{p5.7}
Under the above hypotheses, let $\varPsi:X\to X$ be a transformation that can be written in the form
$$
\varPsi(u)=u+\varPhi(u), \quad u\in X,
$$
where $\varPhi$ is a $C^1$-smooth map such that the image of~$\varPhi$ is contained in~$E$ and
\begin{equation} \label{5.26}
\|\varPhi(u_1)\|\le\varkappa, \quad 
\|\varPhi(u_1)-\varPhi(u_2)\|\le\varkappa\,\|u_1-u_2\| 
\quad \mbox{for all $u_1,u_2\in X$}.
\end{equation}
Then the total variation distance between~$\lambda$ and its image under~$\varPsi$ admits the estimate
\begin{equation} \label{5.22}
\|\lambda-\varPsi_*(\lambda)\|_{\rm var}\le C\varkappa,
\end{equation}
where $C>0$ is a constant not depending on~$\varkappa$. 
\end{proposition}

This proposition is a simple particular case of more general results presented in Chapter~10 of~\cite{bogachev2010}. However, for the reader's convenience, we give a complete proof of Proposition~\ref{p5.7}. 

\begin{proof}
Inequality~\eqref{5.22} needs to be proved for $\varkappa\ll1$, because it is trivial if~$\varkappa$ is separated from zero. Furthermore, since~$\lambda$ has a bounded support, there is no loss of generality in assuming that~$\varPhi$ vanishes outside a large ball.

We first note that if $f\in C_b(X)$, then
$$
(f,\lambda)=\int_F\lambda_F(dw)\int_Ef(v+w)\rho(v)\,dv.
$$
Since $v\mapsto v+\varPhi(v+w)$ is a $C^1$-diffeomorphism of~$E$ for $|\varkappa|\ll1$, we see that
\begin{align*}
(f,\varPsi_*(\lambda))
&=\int_F\lambda_F(dw)\int_Ef\bigl(v+w+\varPhi(v+w)\bigr)\rho(v)\,dv\\
&=\int_F\lambda_F(dw)\int_E\frac{f(v'+w)\rho(\Theta_w(v'))}
{\det\bigl(I+(D\varPhi)(\Theta_w(v')+w)\bigr)}\,dv',
\end{align*}
where $\Theta_w(v')$ denotes the solution of the equation $v+\varPhi(v+w)=v'$ and~$D\varPhi$ stands for the Fr\'echet derivative of~$\varPhi$. 
Combining the above formulas, we get
\begin{equation} \label{5.022}
\delta(f,\lambda):=(f,\varPsi_*(\lambda))-(f,\lambda)=
\int_F\lambda_F(dw)\int_Ef(v+w)\Delta(v,w)\,dv,
\end{equation}
where we set
$$
\Delta(v,w)=
\frac{\rho(\Theta_w(v))}{\det\bigl(I+(D\varPhi)(\Theta_w(v)+w)\bigr)}-\rho(v). 
$$
The $\varkappa$-Lipschitz continuity of~$\varPhi$ implies that 
$$
\bigl|\det\bigl(I+(D\varPhi)(\Theta_w(v)+w)\bigr)^{-1}-1\bigr|
=O(\varkappa)\quad\mbox{for all $v\in E$, $w\in F$},
$$
where $O(\varkappa)$ denotes any function whose absolute value can be estimated by~$C\varkappa$ (uniformly with respect to all other variables). It follows that 
\begin{equation} \label{5.23}
\Delta(v,w)=
\frac{\rho(\Theta_w(v))-\rho(v)}{\det\bigl(I+(D\varPhi)(\Theta_w(v)+w)\bigr)}
+\rho(v)O(\varkappa). 
\end{equation}
Furthermore, using inequalities~\eqref{5.26}, we derive
\begin{equation} \label{5.24}
\|\Theta_w(v)\|\le \tfrac{1}{1-\varkappa}\,\bigl(\|v\|+\varkappa\|w\|\bigr),\quad
\|\Theta_w(v)-v\|\le \tfrac{\varkappa}{1-\varkappa}\,\bigl(\|v\|+\|w\|\bigr).
\end{equation}
Substituting~\eqref{5.23} into~\eqref{5.022} and using inequalities~\eqref{5.24}, we obtain
\begin{align*}
|\delta(f,\lambda)|
&\le\int_F\lambda_F(dw)\int_E|f(v+w)|\,
\frac{|\rho(\Theta_w(v))-\rho(v)|}{\det\bigl(I+(D\varPhi)(\Theta_w(v)+w)\bigr)}
\,dv+O(\varkappa)\\
&\le\|f\|_\infty\int_F\lambda_F(dw)\int_E\bigl|\rho(v)-\rho(v+\varPhi(v+w))\bigr|\,dv
+O(\varkappa).
\end{align*}
Taking the supremum over $f\in C_b(X)$ satisfying $\|f\|_\infty\le1$ and noting that the integrals on the right-hand side can be taken over sufficiently large balls, we obtain
\begin{align*}
\|\lambda-\varPsi_*(\lambda)\|_{\rm var}
&\le\frac12\int\limits_{B_F(R)}\lambda_F(dw)\int\limits_{B_F(R)}
\bigl|\rho(v)-\rho(v+\varPhi(v+w))\bigr|\,dv+O(\varkappa)\\
&\le\frac\varkappa2\int\limits_{B_F(R)}\lambda_F(dw)\int\limits_{B_F(R)}
\int_0^1\bigl|\nabla\rho(v+\theta\varPhi(v+w))\bigr|\,d\theta\,dv+O(\varkappa),
\end{align*}
where we used the first inequality in~\eqref{5.26}. 
It is straightforward to see that the right-hand side of the above inequality can be estimated by~$C\varkappa$. 
\end{proof}

\bigskip
{\bf Acknowledgment.}
The author is grateful to Lihu Xu for stimulating discussion.

\addcontentsline{toc}{section}{Bibliography}
\def\cprime{$'$} \def\cprime{$'$}
\providecommand{\bysame}{\leavevmode\hbox to3em{\hrulefill}\thinspace}
\providecommand{\MR}{\relax\ifhmode\unskip\space\fi MR }
\providecommand{\MRhref}[2]{%
  \href{http://www.ams.org/mathscinet-getitem?mr=#1}{#2}
}
\providecommand{\href}[2]{#2}

\end{document}